%% file: kingBM12.tex
\documentclass[12pt]{article}
\usepackage[ansinew]{inputenc}
\usepackage[small]{caption}

\usepackage{latexsym,amsmath,amsthm,amssymb,amsfonts,amscd}

\usepackage{graphicx}
\usepackage{color}
\usepackage[all]{xy}

\usepackage{epsfig}

\usepackage{graphics}

\date{Draft version --- \today}

\newcommand{\indic}[1]{\mathbf{1}_{\{#1\}}}

\newtheorem{theorem}{Theorem}

\newtheorem{lemma}[theorem]{Lemma}
\newtheorem{coro}[theorem]{Corollary}
\newtheorem{prop}[theorem]{Proposition}
\theoremstyle{definition}

\newtheorem{rmk}[theorem]{Remark}

\newcommand{\eps}{\varepsilon}
\newcommand{\E}{\mathbb{E}}

\newcommand{\La}{\Lambda}

\newcommand{\N} {\mathbb{N}}

\def\mn{\medskip\noindent}

\def\th{^{\text{th}}}

\def\cG{\mathcal G}

\def\cH{\mathcal H}

\def\cK{\mathcal K}

\def\cP{\mathcal P}

\def\PP{\mathbb{P}}
\def\bd{\partial \mathbb{T}}

\begin{document}

\title{Kingman's coalescent and Brownian motion.}

\author{Julien Berestycki$^1$ and Nathana\"{e}l Berestycki$^2$}

\maketitle

\begin{abstract}
We describe a simple construction of Kingman's coalescent in terms of a Brownian excursion. This construction is closely related to, and sheds some new light on, earlier work by Aldous \cite{aldouscond} and Warren \cite{warren}. Our approach also yields some new results: for instance, we obtain the full multifractal spectrum of Kingman's coalescent. This complements earlier work on Beta-coalescents by the authors and Schweinsberg \cite{bbs2}. Surprisingly, the thick part of the spectrum is \emph{not} obtained by taking the limit as $\alpha \to 2$ in the result for Beta-coalescents mentioned above. Other analogies and differences between the case of Beta-coalescents and Kingman's coalescent are discussed.
\end{abstract}

\vfill
\mn 1.  \small{\texttt{julien.berestycki@upmc.fr}. Universit\'e Paris VI - Pierre et Marie Curie, Laboratoire de Probabilit\'es et Mod\`eles Al\'eatoires.}

\mn 2. \small{\texttt{N.Berestycki@statslab.cam.ac.uk}. University of Cambridge, DPMMS, Statistical Laboratory.}

\newpage

\section{Introduction and results}

Let $(\Pi_t,t \ge 0)$ be Kingman's coalescent. This is a Markov process taking its values in the set $\cP$ of partitions of $\mathbb{N}=\{1,\ldots\}$, such that initially $\Pi(0)$ is the trivial partition composed exclusively of singletons, and such that each pair of blocks merges at rate 1. Kingman's coalescent was introduced in 1982 by Kingman in his seminal paper \cite{king82}. A fascinating mathematical object in its own right, Kingman's coalescent is also a cornerstone of mathematical population genetics, serving as the basic model for genealogies in a well-mixed, selectively neutral population.

In this paper we provide a new, simple construction of Kingman's coalescent in terms of a Brownian excursion. We apply this construction to the study of some fine properties of this process, paying particular attention to its behaviour near the beginning of time, when microscopic clusters coalesce into a finite number of macroscopic ones. This phenomenon is known as \emph{coming down from infinity}, and we are able to describe the precise multifractal spectrum corresponding to clusters of atypical sizes at small times. Our construction is closely related to (and, in some sense, almost implicit in) earlier work by Aldous \cite{aldouscond}, Warren \cite{warren} and others (see section \ref{S:discussion} below). However, there are important differences, which will also be discussed in Section \ref{S:discussion} (for instance, the aforementioned application could not have been deduced from these works). This paper complements the picture developed in \cite{bbs2} and \cite{bbl2} on the relation between coalescents, continuum random trees, and Fleming-Viot type particle systems.

We now describe our construction of Kingman's coalescent. Let $(B_t,t\ge 0)$ be an excursion of Brownian motion conditioned to hit level 1. That is, let $\nu$ denote It\^o's excursion measure on the space of continuous excursions $
\Omega^*:= \bigcup_{\zeta>0} \Omega_\zeta,
$
where
$$
\Omega_\zeta:=\{f:[0,\zeta] \to \mathbb{R} \text{ continuous, } f(x)=0 \iff x \in \{0, \zeta\}\}.
$$
Let then $(B_t,0 \le t \le \zeta)$ be a realization of $\nu(\cdot | \sup_{s>0} B_s \ge 1)$. (With a slight abuse of notation, we may consider $B$ to be a function defined on $[0,\infty)$ by declaring $B(s)=0$ for all $s \ge \zeta$.) Let $\{L(t,x)\}_{t\ge 0, x \ge 0}$ denote a jointly continuous version of the local-time process of $B$, and define $$
Z_x := L(\zeta, x), \ \ x \ge 0.
$$
Thus $Z_x$ is the total local time accumulated at level $x$ by the excursion $(B_s, s \ge 0)$. Define a $\cP$-valued process $(\Pi_u, 0 \le u \le 1)$ as follows. Consider the set $\{\eps_i\}_{i=1}^\infty$  of excursions of $B$ above level 1, ordered according to their height: that is, $$\sup_{s>0} \eps_1(s) > \sup_{s>0} \eps_2(s) > \ldots $$ Now, fix $0<u <1$, and consider the set of excursions $\{e_k\}_{1 \le k \le N}$ above level $u$ reaching level $1$, where $N=N(u)$ is the number of such excursions. Note that for every $i \ge 1$, $\eps_i$ belongs to exactly one excursion $e_k$ for some $1 \le k \le N$, and let $\phi_u(i) =k \in \{1, \ldots, N\}$ be this index. Then define $\Pi_u$ by declaring that for every $i,j \ge 1$, $i$ and $j$ are in the same block of $\Pi_u$ if and only if $\phi_u(i)=\phi_u(j)$, that is, if and only if
$\eps_i$ and $\eps_j$ fall into the same excursion $e_k$ for some $ k \le N$. Our main result states that $\Pi_u$ is in fact a time-change of Kingman's coalescent.

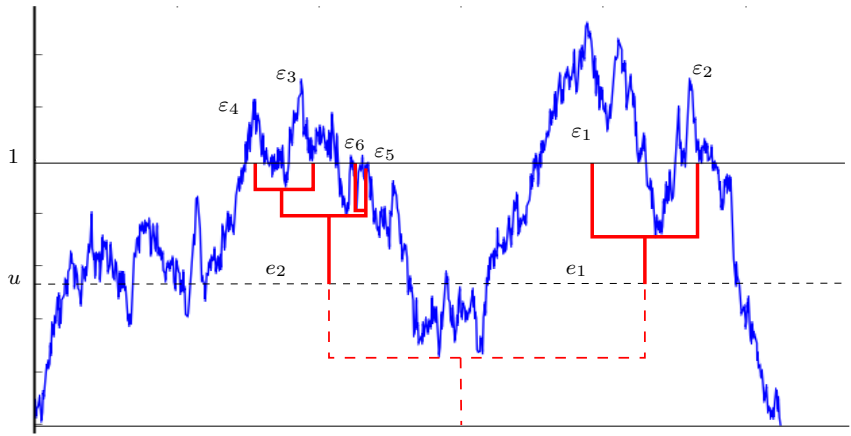
\begin{figure}
 \begin{center}
\leavevmode
\input{construction.pstex_t}
\caption{\ Construction in Theorem 1. In this picture, we have: $\Pi_u=(\{1,2\},\{3,4,5,6\})$.}
\end{center}
\end{figure}

\begin{theorem} \label{T1} The process $(\Pi_{U(t)}, t \ge 0)$ has the same law as Kingman's coalescent, where for all $t>0$,
\begin{equation}
U(t) = \sup\left\{s>0: \int_s^1 \frac4{Z_u}du >t\right\}.
\end{equation}
\end{theorem}

\begin{rmk} \mn
\begin{enumerate}

\item As the reader has surely guessed, the ordering of the excursions $(\eps_i)_{i=1}^\infty$ by their height is not crucial to this result, but is in the spirit of our use of the Donnelly-Kurtz lookdown approach (see section \ref{S:proofs} and \cite{bbs2}).


\item The time-change $\{U(t), t \ge 0\}$ satisfies the following properties: $U(0)=1$, $\lim_{t \to \infty}U(t)=0$, and $U$ is continuous and monotone decreasing.

\end{enumerate}
\end{rmk}

\medskip As promised at the beginning of this introduction, we now give some applications of Theorem \ref{T1} to the study of the small-time behaviour of Kingman's coalescent. Let $(\kappa_t, t \ge 0)$ be Kingman's coalescent, and let
\begin{equation}\label{freqd}
F(t)= \lim_{n \to \infty} \frac1n \# \{1 \le i \le n: \text{ $i$ is in the same block as 1 in $\kappa_t$}\}.
\end{equation}
$F(t)$ is the frequency of the block containing 1 at time $t$, and the existence of the almost sure limit in (\ref{freqd}) (for all times simultaneously) follows from general theory on coalescent processes and exchangeable partitions: see, e.g., Proposition 2.8 in \cite{bertoin}. In the same manner, one may define $F_i(t)$ for all $i \ge 1$ to be the asymptotic frequency of the block containing $i$ at time $t$, thus $F(t)=F_1(t)$ in (\ref{freqd}). The first corollary gives us the behaviour of the typical block size near time zero.
While this result is well known, our proof is new. Along the way we also provide an alternative path to a result of Aldous concerning the asymptotic of the number of blocks with a given size.

\begin{coro} \label{C:typ size}
As $ t \to 0$,  $$\frac{2F(t)}t \overset{d}\longrightarrow  E+E'$$
where $\overset{d}\longrightarrow  $ stands for convergence in distribution, and where $E$ and $E'$ are two independent exponential variables with parameter 1.
\end{coro}

Our second application is, to the best of our knowledge, new. It concerns the existence and the number of blocks (in the sense of Hausdorff dimensions to be specified below) with atypical sizes as $t \to 0$, that is, blocks of size of order $t^\gamma$ with $\gamma \neq 1$. It turns out that for $\gamma<1$ (i.e., for anomalously large blocks) we have to look at a more precise, logarithmic, scale as there are almost surely no blocks whose size will be $t^\gamma$ for small $t$. In particular, the sizes of the smallest and largest blocks at small time $t>0$ are identified.

\medskip For this result, the framework introduced by Evans \cite{evans} is very convenient. Consider a random metric space $(S,d)$, defined as follows. Define a distance $d$ on $\N$ by declaring that for every $i,j \ge 1$, $d(i, j)$ is the time at which the integers $i$ and $j$ coalesce, and let $S$ be the the completion of $\N$ with respect to $d$.
It can be shown that to every ball of radius $t>0$ say, corresponds a unique block of the partition $\kappa_t$. The space $S$ is thus naturally endowed with a measure $\eta$ such that for every $x \in S$ and for every $t>0$, $\eta(x,t):=\eta(B(x,t))$ is the asymptotic frequency of the block of $\kappa_t$ associated with $B(x,t)$.
In this setting, the question mentioned above becomes: are there points $x \in S$ such that $\eta(B(x,t))$ is approximately $t^\gamma $ as $t \to 0$, and if so, what is their Hausdorff dimension? Define for $\gamma > 1$
\begin{equation}\label{small}
S_{\text{thin}}(\gamma)=
 \left\{x\in S:
\limsup_{t\to 0}\frac{\log(\eta(x,t))}{\log t} = \gamma \right\}.
\end{equation}
This set corresponds to points of $S$ with atypically small $\eta(x,t)$. For thick points, consider for all $\theta \ge 0$,
\begin{equation}\label{big}
S_{\text{thick}}(\theta)=  \left\{x\in S: \limsup_{t\to
0}\frac{\eta (x,t)}{t |\log t|} = \theta \right\}.
\end{equation}

\begin{theorem}\label{T:MF spectra} \nopagebreak[4]
\begin{enumerate}
\item[] \nopagebreak[4] \item[1.]
If $0 \le \theta \le 1$ then
$$
\dim S_{\text{\em thick}}(\theta)=1-\theta, \ a.s.
$$
When $\theta =1$, $S_{\text{\em thick}}(\theta)\neq \emptyset $ almost surely, but if $\theta>1$ then $S_{\text{\em thick}}(\theta)$ is empty almost surely.

\item[2.] If $1< \gamma \le
2$ then
$$
\dim S_{\text{\em thin}}(\gamma)=\frac{2}{\gamma}-1, \ a.s.
$$
If $\gamma > 2$ then
$S_{\text{\em thin}}(\gamma)=\emptyset$ a.s. but $S_{\text{\em thin}}(2)\neq
\emptyset$ almost surely.
\end{enumerate}
\end{theorem}

It may be deduced from the above result that the Hausdorff dimension of $S$ is equal to 1 almost surely, a result which was first proved by Evans \cite{evans}.
This result should be compared to Theorem 5 in \cite{bbs2} which describes the multifractal spectrum for Beta$(2-\alpha,\alpha)$ coalescent, $\alpha \in (1,2)$. Kingman's coalescent is a limit case for this family and corresponds formally to the case $\alpha =2$, since the Beta$(2-\alpha, \alpha)$ distribution converges weakly to a Dirac mass at 0 when $\alpha \to 2$. Interestingly, only the ``thin points" side of the spectrum is obtained by taking a limit as $\alpha \to 2$ in that result: at the power-law scale, the ``thick points" part of the spectrum is empty, although the limit as $\alpha \to 2$ exists and is non-degenerate. For instance, here the smallest block turns out to be of order $t^2$ as $t \to 0$, while the largest block is of order $t \log(1/t)$. In the case of Beta-coalescents with parameter $1< \alpha<2$, these quantities are respectively $t^{\alpha/(\alpha -1)^2}$ and $t^{1/\alpha}$, which does not coincide with $t\log(1/t)$ when $\alpha \to 2^-$.

We emphasize that the proof of Theorem \ref{T:MF spectra} is in some sense purely conceptual: all the work consists of identifying Evans' metric space with a certain Galton-Watson tree (which here turns out to be the Yule tree), equipped with its branching measure. Theorem \ref{T:MF spectra} then follows automatically from the results of Peter M\"orters and Nan-Rueh Shieh \cite{ms02, ms04} on the multifractal spectrum of these measures. These results themselves rely on elegant percolative methods introduced by R. Lyons \cite{lyons}.

\section{Previous results and discussion}
\label{S:discussion}


We review some earlier results concerning the relation between Kingman's coalescent and Brownian processes. We start by discussing the ideas contained in \cite{aldous} and \cite{warren}, which are most directly related to the representation of Theorem \ref{T1}.

\subsection{Excursions conditioned on their local time}

The notion of continuum random tree (CRT) was developed by D. Aldous in his seminal papers \cite{aldous1, aldous3}, in which a careful treatment of the correspondence between excursions and continuum random trees is given. Particular attention is paid to the normalized Brownian excursion and to the tree that it encodes through this correspondence.
Early effort bore on the identification of the law of the tree spanned by a finite number of leaves chosen suitably at random. Given an excursion function $e \in \Omega^*$ and $p \in \N$, let $t_1,t_2, \ldots, t_p \in (0, \zeta)$ be pairwise distinct times, and define a planar labelled tree with edge-lengths $ T=T_p(e,t_1,\ldots , t_p)$ as follows.
\begin{itemize}
 \item[-] $T$ contains a root at height 0 and $p$ leaves, with the height of the $k\th$ leaf being $e(t_k).$
 \item[-] The path from the root to the leaves $j$ and $k$ ($1 \le j\neq k \le p$) splits at height $\inf\{e(s), s\in (t_{j}, t_{k}) \}.$
 \item[-] At each branch point, branches are labelled ``left" and ``right".
\end{itemize}
It is further assumed that all branch points are distinct (i.e., branch points are binary). Thus $T$ has $p-1$ branch points. Le Gall \cite{legall93} (see his Theorem 3) first identified the distribution of $T_p(e,t_1,\ldots , t_p)$ on the space of trees when $e$ is a normalized Brownian excursion and the $t_k$ are independent and uniform random variables on $[0,\zeta(e)]$ (and further gives the conditional distribution of $e$ conditionally given the tree $T$). Later on, Aldous \cite{aldouscond} gave a decomposition result for this distribution by identifying the conditional distribution of $T$ given the local time profile $(\ell_x, x \ge 0)$ where $\ell_x$ is the total local time accumulated at $x$ by $e$. This conditional distribution is constructed from a certain inhomogeneous (i.e., time-dependent) coalescent process whose main ingredient is Kingman's coalescent. More precisely, this process is defined as follows. First, the height of the leaves $(x_1, \ldots, x_p)$ are i.i.d. samples from $\ell_x\ dx$ (which is a probability distribution because $\zeta(e)=1$ almost surely). Then, thinking of time as running from  $x= \infty$ down to $x=0$ and the heights $x_i$ as being the birth times of particles, the law of $T$ is such that as $x$ decreases, the particles that are present merge pairwise, each independently at rate $4/\ell_x$. In other words, the conditional distribution of $T$ given $\ell$ can be thought of as \emph{an inhomogeneous coalescent with immigration rate $\ell_x$ and coalescing rate $4/\ell_x.$}  This result may be regarded as the continuum random tree counterpart to Perkins' disintegration theorem in the context of superprocesses \cite{perkins91}. The proof of Aldous is based on a discrete approximation to the continuum random tree, but soon afterwards Warren \cite{warren} gave two alternative, direct proofs of Aldous's result. The tools used in those arguments are closely related to earlier work by Le Gall \cite{legall93} and Warren and Yor \cite{warrenyor}.

It thus comes as no surprise that one can embed Kingman's coalescent into a Brownian excursion. However, we emphasize that our construction is rather different in that the tree which we consider is spanned by vertices at distance 1 from the root, rather than by leaves selected uniformly at random. Moreover, it seems difficult to use the description of $T$ given above to deduce results about Kingman's coalescent. On the other hand, as the reader will see, here these conclusions will follow in a straightforward fashion once Theorem \ref{T1} is proved. Finally, we believe that the computations leading up to Theorem \ref{T1} are new, and hope that the reader will find them interesting in themselves.

\subsection{Analogies and differences with Beta-coalescents}

The present paper complements earlier results of \cite{bbs1, bbs2}, which focuses on small-time properties of Beta -coalescents with parameter $\alpha \in (1,2).$ Beta-coalescents are a family of $\Lambda$-coalescent (i.e., coalescents with multiple collisions), where the measure $\Lambda$ is the density of a Beta$(2-\alpha,\alpha)$ random variable. (The interested reader may consult \cite{bbs2} and references therein.) Kingman's coalescent constitutes a formal limit case for the Beta-coalescents as $\alpha \to 2$ since the Beta$(2-\alpha, \alpha)$ converges weakly to the Dirac mass in zero as $\alpha \to 2$. It was proved in \cite{bbs2} that for all $1<\alpha<2$, Beta-coalescents can be embedded in continuum stable random trees associated with $\alpha$-stable branching processes rather than Brownian motion. (An excellent introduction to continuous random trees can be found in \cite{duleg}). When $\alpha=2$ the $\alpha$-stable branching process is the Feller diffusion, which is closely related to Brownian motion. It is therefore natural to suspect a relation between Kingman's coalescent and Brownian motion.

However, we emphasize that there is an essential difference between Theorem \ref{T1} here and Theorem 1 in \cite{bbs2}.
The analogue of Theorem 1 in \cite{bbs2} is the following. Let $(B_s, s \ge 0)$ be a reflecting Brownian motion and let $\tau_1 = \inf \{ t >0 : L(t,0) >1\}$, where $L(t,x)$ is the joint local time process of $B$. Let $v>0$ be such that $v< \sup_{s\le  \tau_1} B(s)$, and for all $0 \le u \le v$, define a partition $\Pi_u^v$ in exactly the same way as in the construction given above Theorem \ref{T1}, except that where we used the level 1 we now use level $v$: thus, here the excursions $(\eps_i)_{i=1}^\infty$ which we consider are those above level $v$ (instead of 1) and the $e_k$ are the excursions above $u$ that reach $v$.
Therefore, using this notation, the partition $\Pi_u$ defined for Theorem \ref{T1} is simply $\Pi_u^1$ (with, however, the difference that here $B$ is not a single excursion).

Now, for all $t>0$, let
$$
V(t) := \inf \left\{ s>0 : \int_0^s 4Z_u^{-1} du > t  \right\}
$$
where $Z_x= L(\tau_1, x)$ for all $x\ge 0$. (Note that $V(t) < \sup_{s \le \tau_1} B(s)$). Fix $T>0$ and for $0 \le t \le T$, consider the partition
$$
\hat \Pi_t := \Pi_{V(T-t)}^{V(T)}.
$$
We can prove the following result:
\begin{prop} \label{P1}
$(\hat \Pi_s, 0 \le s \le T)$ has the same law as $(\kappa_s, 0 \le s \le T)$, Kingman's coalescent run for time $T.$
\end{prop}

Observe that the main difference between Proposition \ref{P1} and Theorem 1 is that here the tree is spanned by vertices at a random distance $V(T)$ from the root of the continuum random tree, while these vertices are at a deterministic distance 1 in Theorem \ref{T1}. This was the source of considerable technical difficulties in \cite{bbs2}. As the reader will see, the proof of Theorem \ref{T:MF spectra} is much more straightforward than that of Theorem 5 in \cite{bbs2}.

A direct proof of Proposition \ref{P1} is provided in Section \ref{S:proofs}. However, it could also be deduced from Perkins' disintegration theorem together with Le Gall's excursion representation of Dawson-Watanabe processes. For this, one must use the fact that the empirical measure for an exchangeable system of particles fully determines the law of the individual trajectories.

\medskip In light of Theorem \ref{T1} and the above discussion, one can wonder if, conversely, Theorem \ref{T1} could not be extended to Beta-coalescents. Interestingly, we have rather strong evidence that this cannot be true. Assuming that the construction was valid for a Beta-coalescent, we would conclude that the number of blocks lost at every collision would be an i.i.d. sequence. Indeed, from results in \cite{duleg}, one can see that the tree spanned by vertices at a deterministic level (the ``reduced tree") forms a time-change of a Galton-Watson process. However, in that case the distribution of the number of blocks involved at every collision must agree with the distribution in equation (10) of \cite{bbs1}. Simulations indicate that this is not the case for the final collision. However, we are not aware of a rigorous proof that this cannot be true, although there are also some theoretical arguments in that direction.
We note that the essential reason for which the proof of Theorem \ref{T1} breaks down in the case $\alpha <2$ is that Brownian motion possesses an extra symmetry property (reflection about a horizontal line) compared to the height process of other stable continuous-state branching processes.

\subsection{More on duality.}

Theorem \ref{T1} brings to mind another result which relates Kingman's coalescent to the behaviour of Yule processes. (As the reader will see in section \ref{S:preuve du thm}, Yule processes are indeed embedded in the construction of Theorem \ref{T1}). Consider the jump chain of Kingman's coalescent i.e., the chain $X(1), \ldots,
X(n), \ldots$ where $X(n)$ is the element of the $n\th$ simplex given by the block frequencies of Kingman's coalescent when it has $n$ blocks. In \cite{bg}, Bertoin and Goldschmidt show that this chain $X$ can be related to the fragmentation chain obtained by taking a Yule process $(Y_t, t \ge 0)$ and conditioning on $\{W=w>0\}$, where $W=\lim_{t\to\infty} e^{-t} Y_t$ almost surely. This fragmentation process $(G_t, t \ge 0)$ is defined by considering how the mass $w$ is being split between the children of an individual when it dies. More precisely, the two children of this individual each give rise to their own independent Yule process $Y^{(1)}$ and $Y^{(2)}$, say. To each one we can associate the corresponding random variable $W^{(1)}$ and $W^{(2)}$. Note that if $\tau$ is the time of the split, we must have $W= e^{-\tau}W^{(1)} + e^{-\tau} W^{(2)}$. Define the fragmentation process $(G_t, t\ge 0)$ by saying that at time $\tau$, the mass $W$ splits into $e^{-\tau} W^{(1)}$ and $e^{-\tau} W^{(2)}$, and so on as time evolves. If we do not condition on $\{W=w\}$, then $G$ is a fragmentation in the sense that fragments evolve independently with individuals fragments behaving as a rescaled copy of the original process. This is what is usually called a \emph{homogeneous fragmentation process} (see, e.g., \cite[Section 3.1]{bertoin}) but note that here the total mass is random. It is then natural to ask what happens when we condition on $\{W=w\}$ to keep the total mass of this fragmentation deterministic. Theorem 3.1 of \cite{bg} then states that if $(N_t,t\ge 0)$ is an independent Poisson process, then conditionally on $\{W=w\}$,
\begin{equation}\label{bg}
\{G_{\log(1+t)},t\ge 0\} \overset{d}= \{wX(N_{wt}),t\ge 0 \}.
\end{equation}
That is, up to a time-change, the fragmentation $G$ is the time-reversal of the jump chain of Kingman's coalescent. It is an open problem to decide whether a similar representation exists for the case of Beta-coalescents. We indicate that Christina Goldschmidt \cite{cg} has recently computed explicitly the Martin boundary of the continuous Galton-Watson process $(Y_t, t \ge 0)$ associated with the reduced tree at level 1 of a stable continuum random tree with index $\alpha$. This can be used to describe explicitly the behaviour of $Y^{(w)}$, where $Y^{(w)}$ denotes the process $Y$ conditioned on $\{W=w\}$, (here again, $W= \lim_{t\to \infty} e^{-\lambda t} Y_t$ and $\lambda = 1/(\alpha -1)$).

In particular, from her result it follows that the transition rates of $Y^{(w)}$ are not independent of the current state of $Y^{(w)}$. Since it turns out that $W=Z_1$ almost surely, ($Z_1$ being the quantity of local time of the stable CRT at level 1), it follows that conditionally on $\{Z_1=z\}$, the number of children of individuals in the tree is not i.i.d. In particular, the above objection against the extension of Theorem \ref{T1} does not hold here.



\section{Brownian construction of Kingman's coalescent.}

\label{S:proofs}

\subsection{Proof of Theorem \ref{T1}.}

\begin{proof}
The proof we give is based on a few simple calculations with excursion theory. For basic definitions and facts, we refer the reader to Chapter XII of \cite{revuz-yor}.

Fix $0<t<1$, and recall our notation $Z_x= L(\zeta, x)$ for all $x \ge 0$, where $L(t,x)$ is the joint local time of a Brownian excursion $(B_s, s \ge 0)$ conditioned to exceed level 1.  Define a filtration
\begin{equation}\label{G}
\{\cG_u = \sigma(Z_s, u \le s \le 1)\}_{0 \le u \le 1}.
\end{equation}
In parallel, define a family of $\sigma$-algebras $\{\cH_u\}_{0\le u \le 1}$ by putting
\begin{equation}
\label{H}
\cH_u= \sigma\left(B_{\alpha^u(s)}, s \ge 0\right)
\end{equation}
where for all $s>0$:
$$
\alpha^u(s):= \inf\left\{t>0: \int_0^t \indic{B_\tau>u} d\tau >s\right\}.
$$
In words $\cH_u$ contains all the information about the trajectory above level $u.$
It is a tedious but easy exercise to see that $\{\cH_{1-t}\}_{0\le t \le 1}$ is a filtration and $\cH_u \supseteq \cG_u$ for all $0 \le u \le 1$. Let $\delta >0$, and let $u=1-t$. Define the event $A_{k,j}(\delta)$ as follows. Call $N$ the number of excursions above $u$ which reach 1, and let $(e_1,\ldots e_N)$ be these excursions with an order given by their height (as usual, we view each $e_i$ as an element of $\Omega^*$). Let $1\le k,j \le N$ and let
$$
A_{k,j}(\delta):= \{\text{excursions $e_k$ and $e_j$ have coalesced at level $u-\delta$}\},
$$
where we say that $e_k$ and $e_j$ have coalesced at level $u-\delta$ if $e_k$ and $e_j$ are part of a single excursion above level $u-\delta$ which reaches level 1.

To alleviate notations we also write $A(\delta)$ when there is no risk of confusion.
Theorem \ref{T1} follows from the following claim:
\begin{equation}\label{claim1}
\PP( A_{k,j}(\delta) | \cG_u ; N \ge \max(j,k)) = \frac4{Z_u} \delta + o(\delta) , \ \ a.s.
\end{equation}
Note that to every excursion $e_i$ with $1\le i \le n$, we may associate a starting time $t_i \in [0, \zeta]$ such that the excursion starts at time $t_i$, that is: $B(t+t_i) = u+e_i(t)$,
for all $t \le \zeta(e_i)$. Assume without loss of generality that $t_j < t_k$. Among all excursions \emph{below} level $u$, consider the subset of those excursions $(e'_i, i \ge 1)$ which \emph{separate} $e_k$ and $e_j$, that is, whose starting time $t'_i$ lies strictly in the interval $(t_j , t_k)$ (see Figure \ref{fig1}).
\begin{figure}
\includegraphics[scale=.75]{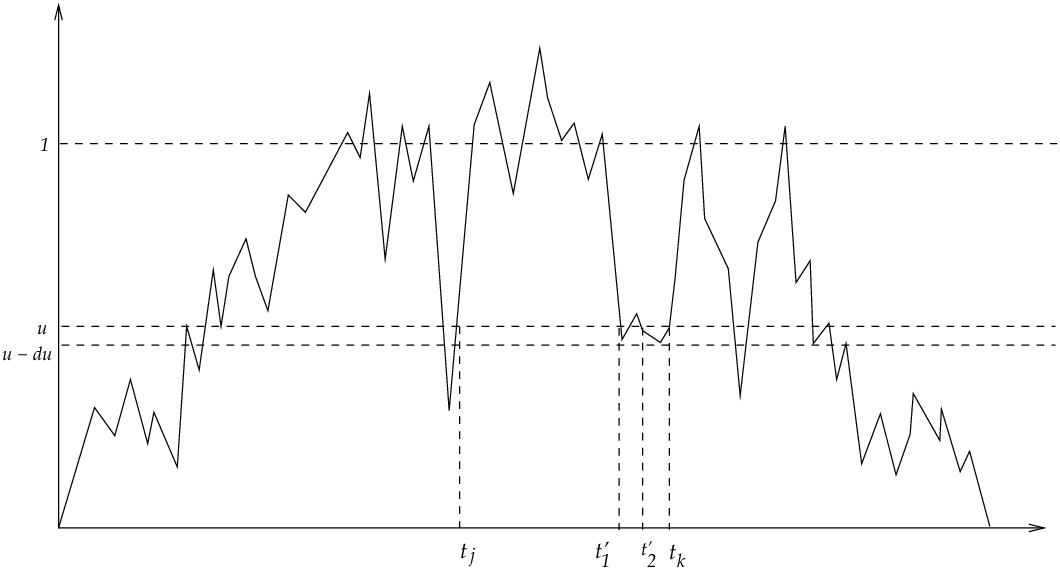}
\caption{Between levels $u$ and $u-du$, excursions $e_j$ and $e_k$ (which start respectively at times $t_j$ and $t_k$) coalesce. For this to occur, the excursions below $u$ starting at time $t'_1$ and $t'_2$ must not reach below $u-du$.}
\label{fig1}\end{figure}
Observe that for $e_j$ and $e_k$ to coalesce by level $u-\delta$, it is both necessary and sufficient that:
\begin{equation}
\inf_{s>0} e'_i(s) > -\delta, \ \ \text{for all $i\ge 1$}.
\end{equation}
Since $L(t,u)$ may only increase at times $t$ such that $B_t=u$, it follows that throughout $e_j$, the local time at level $u$ is constant, let $\ell_j$ be this quantity (thus, $\ell_j = L(t_j,u)$). Similarly, let $\ell_k$ be the local time at level $u$ throughout excursion $e_k$, and let $\ell= |\ell_j - \ell_k|$. Then we claim that, conditionally on $\ell_j$ and $\ell_k$, $A(\delta)$ has probability $\exp(-\ell/(2\delta))$. Indeed, by excursion theory and elementary properties of Brownian motion, we have that, conditionally on $\cH_u$,
$$
\sum_{i \ge 1} \delta_{(\ell'_i, e'_i)}
$$
is a Poisson point process on $[\ell_j, \ell_k] \times \Omega^*$ with intensity $\indic{[\ell_j, \ell_k]}d\ell \otimes \nu^u(de)$, where $d\ell$ is the one-dimensional Lebesgue measure, and $\nu^u$ is It\^o's excursion measure restricted to negative excursions $e \in \Omega^*$ such that $\inf_{s>0} e(s) >-u$. It follows that
$$\sum_{i\ge 1} \delta_{\left(\ell'_i, \inf_{s>0}e'_i (s)\right) }$$
is a Poisson point process on $[\ell_i,\ell_j] \times \mathbb{R}_-$ with intensity $d\ell \indic{\ell \in [\ell_i, \ell_j]} \otimes \indic{- u \le h \le 0} dh/2h^2 $, where $z:= Z_u$ (see, e.g., (2.10) in Chapter XII of \cite{revuz-yor}).
Thus,
$$
\PP(A(\delta)| \cH_u) = \exp(- |\ell_j - \ell_k|/(2\delta)).
$$
Now, observe that, conditionally on $\cG_u$,
$$\sum_{i\ge 1} \delta_{\left(\ell_i, \sup_{s>0}e_i (s)\right) }$$
is a Poisson point process on $[0,z] \times \mathbb{R}_+$ with intensity $d\ell\indic{\ell \le z} \otimes \indic{h\ge 1-u} dh/2h^2$, because $e_i$ are precisely the set of excursions above $u$ that reach level 1. Therefore, using elementary properties of Poisson processes, conditionally on $\cG_u$, and conditionally on $\{N \ge \max(j,k)\}$ (which we must assume if we are to talk about the coalescence of $e_j$ and $e_k$), then $\ell_j$ and $\ell_k$ are two independent random variables uniformly distributed on $[0,z]$ (this does not depend on the precise intensity of the point process; all that is required is that the time-intensity is the Lebesgue measure and the space-intensity has no atom). In particular, conditionally on $\cG_u$ and $N \ge \max(j,k)$, $\ell:= |\ell_j - \ell_k|$ may be written as: $$\ell=z|U-V|$$
where $U,V$ are uniform random variables on $(0,1)$ and $z=Z_u$. Putting these pieces together, we obtain:
\begin{align*}
\PP(A(\delta) | \cG_u; N \ge \max(j,k)) & = \mathbb{E}[\exp(- z|U-V|/(2\delta))], \ \ a.s. \\
&= \int_0^1 \int_0^1 dx dy \exp\left(-\frac{z|x-y|}{2\delta}\right) \ \ a.s. \\
&= 2\int_0^1 dy \int_0^ydx \exp\left( -\frac{z(y-x)}{2\delta}\right) \ \ a.s. \\
&= 2\int_0^1 dy \exp\left(-\frac{zy}{2\delta}\right) \int_0^y dx \exp\left(\frac{zx}{2\delta}\right)\ \ a.s. \\
&= 2\int_0^1 dy \exp\left(-\frac{zy}{2\delta}\right) \frac{2\delta}z (\exp\left(\frac{zy}{2\delta}\right) -1) \ \ a.s.\\
&= \frac{4\delta}z \int_0^1 dy [1- \exp\left(-\frac{zy}{2\delta}\right)] \ \ a.s. \\
&= \frac{4\delta}z - \frac{8\delta^2}{z^2}(1-\exp\left(-\frac{z}{2\delta}\right))\ \ a.s. \\
& = \frac{4\delta}z +o(\delta) \ \ a.s.
\end{align*}
as $\delta \to 0$. This is precisely (\ref{claim1}), and so this implies Theorem \ref{T1}.
\end{proof}

\subsection{Proof of Proposition \ref{P1}}

The proof of Proposition \ref{P1} involves similar ideas to the proof of Theorem \ref{T1} above. However, the calculations become, somewhat surprisingly, slightly more complicated. In particular, there is a remarkable cancellation towards the end, which illustrates the following fact. Roughly speaking, we try to compute the rate at which the $k\th$ highest excursion at a current level $u>0$ splits to give the $k\th$ and $j\th$ highest excursions, with $j\ge k$. We are trying to show that this rate is $4/Z_u$ and is in particular independent of $k$ and $j$. This may seem hard to believe at first: when splitting, it is easier to create an
excursion of small size rather than large. However, the excursion heights accumulate near zero and therefore creating an interval whose size falls exactly between the sizes of the $n$th smallest and $n+1$th smallest intervals also becomes harder when $n$ tends to infinity.
These two effects exactly compensate each other and imply the aforementioned result!

\medskip Fix $k,j\ge 1$ and assume that $k\le j$. Recall that here $(B_s, s \le \tau_1)$ is no longer a single Brownian excursion but a collection of excursions which accumulate one unit of local time at level 0. We shall still denote by $Z_t=L(\tau_1,t)$ the
total local time accumulated at level $t$ by $(B_s,0\le s \le \tau_1)$. Recall
that, by the Ray-Knight theorem (Theorem (2.3) in Chapter XI of \cite{revuz-yor}), $(Z_t, t\ge 0)$ is the Feller diffusion, defined by:
$$
\begin{cases}
dZ_t = \sqrt{Z_t}dW_t\\
Z_0=1
\end{cases}
$$
where $(W_t, t\ge 0)$ is a standard Brownian motion.
Let
\begin{equation}\label{K}
\cK_u= \sigma\left(B_{\gamma^u(s)}, s \ge 0\right)
\end{equation}
where for all $s>0$:
$$
\gamma^u(s):= \inf\left\{t>0: \int_0^t \indic{B_\tau\le u} d\tau >s\right\}.
$$
In words $\cK_u$ contains all the information about the trajectory of $B$ below level $u.$
Just as in the proof of Theorem \ref{T1}, $\{\cK_{t}\}_{t>0}$ is a filtration. We wish to
compute the rate at which the $j\th$ excursion ``looks down'' on the
$k\th$ excursion to adopt its label. We claim that, independently
of $k$ and $j$, and independently of $\cH_t$ (defined in (\ref{H}))
the infinitesimal rate at which this happens is:
\begin{equation}
\label{claim2} \text{rate ($j$ looks down on $k$)} = \frac4{Z_t}.
\end{equation}
Here and in what follows, the event that $j$ looks down on $k$ means the following thing. Define a process $\{\xi_i(t), t \ge 0\}_{i\ge 1}$ as follows. Initially, $\{\xi_i(0)\}_{i\ge 1}$ are i.i.d uniform $(0,1)$ random variables. We think of $\xi_i(0)$ as the label of the $i\th$ highest excursion of $(B_s, 0 \le s \le \tau_1)$ above 0. As time evolves, $\xi_i(t)$ keeps this interpretation, with $\xi_i(t)$ being the label of the $i\th$ highest excursion above level $t$ of $(B_s, 0\le s \le \tau_1)$. The rule of evolution is that when the level $t$ increases, an excursion may split into two parts. The highest of these two parts is necessarily still ranked $i\th$, while the second part has a larger rank $j>i$. In that case, we say that \emph{$j$ looks down on $i$}. When this occurs, $\xi_j(t^+)$ becomes $\xi_i(t^-)$, while for $k \ge j$, what was previously the $k\th$ excursion now becomes the $(k+1)\th$ excursion: thus in that case, we define $\xi_{k+1}(t)= \xi_k(t^-)$. Visually, $j$ adopts the label of $i$ (it looks down on it) and all labels corresponding to excursions with a higher rank are pushed up by 1. Thus (\ref{claim2}) tells us that, in the terminology of Donnelly and Kurtz, $\{\xi_i(t)\}_{i\ge 1}$ is a lookdown process with infinitesimal rate $4/Z_t$ at time $t$, conditionally on $(Z_t, t \ge 0)$. We do not need the full definition of these processes, which is somewhat complicated, and can be found, for instance, in Chapter V of \cite{etheridge}. Thus (\ref{claim2}) is exactly the Brownian analogue of Theorem 14 in \cite{bbs2}, where further details on this approach can be found. In particular, Proposition \ref{P1} follows directly from (\ref{claim2}) and Lemma 5.6 in \cite{etheridge}.

We give a direct proof of (\ref{claim2}) based on calculations with
It\^o's excursion measure. Fix a small $\delta>0$. For all $t\ge 0$, let $(e_{t,i})_{i\ge 1}$ be the excursions of $B$ above level $t$ ordered by their height (recall that $e_{t,i} \in \Omega^*$ for all $t\ge 0$, $i \ge 1$). For $k\ge 1$, let
$$
M_k(t) = \sup_{s>0} e_{t,k}(s).
$$
A simple analysis of the construction shows that in time $\delta$
(that is, between times $s$ and $t:=s+\delta$), the $j\th$ excursion
looks down on the $k\th$ excursion if there is a local minimum
located within the $k\th$ highest excursion at level $s\le m \le
t=s+\delta$ such that of the two branches going out of this local
minimum, one has a height equal to $M_k(t^-)=M_k(t)$ and the other
has a height $h_\text{new}$ such that
\begin{equation}
\label{cond} M_j(t^-) <h_\text{new}<M_{j+1}(t^-).
\end{equation}
Let $(e_k(x), 0\le k \le \zeta)$ be this excursion. Given $\{M_k(t)=m\}$, $e_k$ has the law of an It\^o excursion conditioned on
the event:
$
\{\sup_{x\in (0,\zeta)} e_k(x) =m\}.
$
For $m>0$, let $Q^{(m)}$ denote the probability measure on
$C([0,\infty))$ defined by the conditioning of $\nu$ given that
$\{\sup_{s>0} e(s) = m\}$.
Define:
\begin{align*}
T^1_\delta&= \inf\{x>0: e_k(x)>\delta\},\\
T_M &=\text{ the unique $x\ge 0$ such that } e_k(x)=M_k(t)=m,\\ 
T^2_\delta&=\inf\{x>T_M:e_k(x)<\delta\}.
\end{align*}
We note that (\ref{cond}) occurs if and only if the event
$A_{M_{j+1},M_j} \cup B_{M_{j+1},M_j}$ occurs, where
$$
A_{x,y}:=\left\{\exists T^1_\delta<s<T_M: x <  \sup_{r\in [T^1_\delta, s]} e_k(r)<y \text{ and }
e_k(s)<\delta\right\}
$$
and
$$
B_{x,y}:=\left\{x < \sup_{r\in [T^2_\delta, \zeta]} e_k(r)<y\right\}.
$$
For future reference, we also let $A=A_{M_{j+1},M_j}$ and $B= B_{M_{j+1},M_j}$.
\begin{figure}
\input{event_AB.pstex_t}
\caption{Events $A_{x,y}$ on the left and $B_{x,y}$ on the right}
\end{figure}
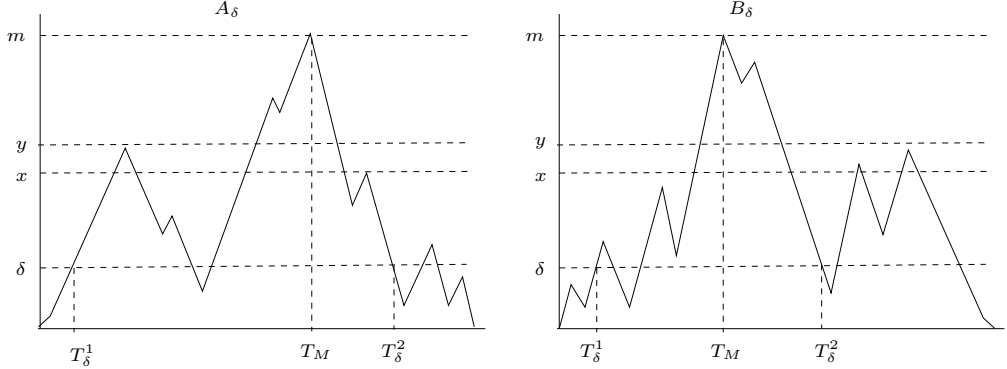

We claim that
\begin{equation}\label{Ecomp}
\PP(A|\cK_t)= \PP(B | \cK_t)= \frac{2\delta}{Z_t} +o(\delta) , \ \  a.s.
\end{equation}
and that $\PP(A\cap B| \cK_t)=O(\delta^2), \ a.s.$, so that, up to terms of order
$o(\delta)$, $\PP(A\cup B|\cK_t )= 4\delta/Z_t , \ a.s.$,
as required in (\ref{claim2}).
By time-reversal symmetry of the Brownian excursion, it is clear that $\PP(A|\cK_t)= \PP(B | \cK_t)$ so we only do the calculations in the case $B.$

\medskip The key idea to compute $\PP(B)$ is to use the Markov property at time $T^2_\delta$. (In fact an additional step is required,
since under this conditioning $e_k$ is not exactly Markov). Let
$m>0$ and let $0<x<y<m$ be two fixed levels (we think of $y=M_j$
and $x=M_{j+1}$).
We claim that $\PP(B)$ can be computed as follows:
\begin{equation}\label{P(B)1}
Q^{(m)}(B_{x,y})=\frac{\delta}x\left(1-\frac{x}y\right) +o(\delta)
\end{equation}
Indeed, consider the unconditional problem first. Recall that the part of the trajectory of a Brownian excursion that comes after $T^1_\delta$ is simply a Brownian motion (started from $\delta$) run until it hits zero.
Let $\eta>0$ and let $M$ be the event that, starting
from $\delta$, it reaches level $m$
and from there returns to
zero without reaching further than $m+\eta$. Starting from $\delta,$ the probability of reaching $m$ before hitting $0$ is $\phi_\delta(m)= \delta/m$ by the standard gambler's ruin probability. Once there, the probability of returning to 0 before hitting $m+\eta$ is $\eta/(m+\eta).$
Thus
\begin{equation}
\label{Ecomp2} \PP^\delta(M)= \phi_\delta(m)\eta/m +o(\eta)
\end{equation}

On the other hand, the event that, started from $\delta$, both $M$
and $B_{x,y}$ occur can be expressed, by using the strong Markov
property at time $T^2_\delta$, as:
\begin{equation}
\label{Ecomp3} \PP^\delta(M\cap B_{x,y})= \phi_\delta(m)
\frac{\eta}{m-\delta}\frac{\delta}x\left(1-\frac{x}y\right) +o(\eta)
\end{equation}
In the above expression, the second term expresses the fact that
the process, having reached distance $m$, only has to return to
distance $\delta$ before reaching $m+\eta$. The third term corresponds
to hitting $x$ before returning to zero, while the fourth and last
term corresponds to the probability that, having reached level
$x$, the process $e$ reaches 0 before $y$. All these calculations
involve the same gambler's ruin probability already used above.
Combining (\ref{Ecomp2}) with (\ref{Ecomp3}), we obtain as
claimed, up to terms of leading order $\delta$:
\begin{equation}
\label{Ecomp4} Q^{(m)}(B_{x,y}) =
\frac{\delta}x\left(1-\frac{x}y\right) +o(\delta).
\end{equation}
Here the term in $o(\delta)$ depends only on $m=M_{k}(s)$. Rewriting
(\ref{Ecomp4}), we find:
\begin{equation}
\label{P(B)2}
Q^{(m)}(B_{x,y})=\delta\left(\frac1x-\frac1y\right)+o(\delta).
\end{equation}
Taking the expectation:
\begin{eqnarray}
Q^{(m)}(B)&=&\delta
E\left(\frac1{M_j(t)}-\frac1{M_{j+1}(t)}\right) +o(\delta) \nonumber \\
&=& \delta\large[\E(1/M_j(s))-\E(1/M_{j+1}(s))\large] +o(\delta).\label{P(B)3}
\end{eqnarray}
Now, observe that by the Markov property of the underlying
continuous random tree, we have above level $s$ a Poisson point
process of excursions with intensity It\^o's excursion measure $n$,
run for a local time $z=Z_s$. Thus, the number $N_{>a}$ of
excursions higher than some fixed level $a>0$ is distributed as:
$$
N_{>a}=\text{Poisson} (z/2a)
$$
since $n(\sup e >a) = 1/2a$. Therefore,
\begin{eqnarray*}
\PP(M_j(s)>a)&=& \PP(N_{>a}\ge j)\\
&=& 1-e^{-z/2a}-\ldots -e^{-z/2a}(z/2a)^{j-1}/(j-1)!
\end{eqnarray*}
from which it follows that $M_j$ has a density equal to
$$
\PP(M_j\in(a,a+da))=\frac{e^{-z/2a} \left(z/2a\right)^j}{a(j-1)!}
$$

This implies:
\begin{eqnarray*}
\E(1/M_j)-\E(1/M_{j+1})&=& \int_0^{\infty} \frac1a e^{-z/2a} \left[(z/2a)^j/a(j-1)! -(z/2a)^{j+1}/a(j)!  \right] da\\
&=&\frac1{j!}\int_0^\infty \frac1a e^{-z/2a} (z/2a)^j \left[ j/a - z/2a^2 \right] da\\
\end{eqnarray*}
using the change of variable
$u=z/2a$, we obtain:
\begin{eqnarray*}
\E(1/M_j)-\E(1/M_{j+1})&=& \frac{1}{j!} \left( -\int_0^\infty e^{-u}u^{j} \frac{2u}{z} \frac{2ju}{z} \frac{z}{2u^2} du +  \int_0^\infty e^{-u}u^j\frac{2u}{z} du \right) \\
&=& \frac{2}{zj!} \int_0^\infty e^{-u}u^{j+1}
du - \frac{2}{z(j-1)!} \int_0^\infty e^{-u}u^j du\\
&=& \frac2z\left( \frac{(j+1)!}{j!} - \frac{j!}{(j-1)!}\right)\\
&=& \frac2z.
\end{eqnarray*}
The cancellation of terms involving $j$ is quite remarkable.
Since $z=Z_t$, we have thus proved the claim
(\ref{Ecomp}):
$$
\PP(B)=2/Z_t.
$$
As explained before, this proves the claim (\ref{claim2}).


\section{Proof of Corollary \ref{C:typ size}}

Let $(B_s, 0 \le s \le \zeta)$ be a Brownian excursion conditioned to reach 1, and let $(\kappa_t,t \ge 0)$ be the realization of Kingman's coalescent described in Theorem \ref{T1}. Note that by Theorem \ref{T1}, one can obtain the frequencies of the blocks of $\kappa_t$ by simply considering the excursions of $B$ above level $U(t)$ with a mass proportional to the quantity of local time that they accumulate at level 1. Thus, to prove Corollary \ref{C:typ size}, the first step consists of computing the number of excursions above a given level $u<1$ which carry a given amount of local time at 1. For $0<u <1$, and $x \ge 0$, let $N(u,x)$ be the number of excursions above level $1-u$ that carry a local time at 1 greater than $x$.
\begin{lemma} \label{L:exc} We have the almost sure convergence:
\begin{equation}\label{cd0}
\frac{2u}{Z_1} N(u,2ux) \longrightarrow e^{-x}
\end{equation}
for all $x\ge 0$ as $u \to 0$.
\end{lemma}

\begin{proof}
To see this, observe that by excursion theory, $N(u):= N(u,0)$ is a Poisson random variable with mean $Z_{1-u}/(2u)$. Moreover, conditionally on $N(u,0)= n$, the $n$ excursions that reach level 1 are i.i.d. realizations of the It\^o measure conditioned to exceed level $u$. Now, it is well-known that such an excursion, upon reaching $u$, behaves afterwards as Brownian motion killed when returning to 0,
so by standard properties of Brownian motion local time (see Proposition 4.6 chapter VI in \cite{revuz-yor}) we have that the amount of local time accumulated by such an excursion at level $u$ is exponentially distributed with mean $2u$.
Therefore, by Poisson thinning, conditionally on $Z_{1-u}=z$,
\begin{equation}\label{poisson}
N(u,x) \overset{d}= \text{Poisson}\left(\frac{z}{2u} e^{-x/(2u)}\right).
\end{equation}
Considering excursions with local time greater than $2ux$, let
$$X_u= \frac{2u e^x}{Z_{1-u}} N(u,2ux).$$
Note that, by almost sure continuity of $Z_s$ near $s=1$, it suffices to prove that $X_u \longrightarrow 1$ almost surely as $u \to 0$.
Next we recall the following standard Chernoff bound for a Poisson random variable $Y$ with parameter $\mu$: for all $h \ge 0$,
$$
\PP(Y>\mu ( 1+\delta)) \le e^{-h\mu ( 1+\delta)} \E ( e^{hY}) = \exp(\mu(e^{h} -1 -h) - h\mu \delta).
$$
Since $e^{h} - 1 - h\sim h^2 /2$ for $h \to 0$, we see that $-\lambda = e^{h} - 1 -h - h \delta <0$ for sufficiently small $h$. Thus, using (\ref{poisson}), we are led to the estimate:
$$
\PP(|X_u-1|> \delta|Z_{1-u}=z) \le C e^{-\lambda z /2u}
$$
for some $\lambda >0$ and some $C>0$.
Taking the expectation in the above, we get
$$
\PP(|X_u-1|> \delta) \le \E(Ce^{-\lambda Z_{1-u} /2u}).
$$
For $k\ge 1$, let $u_k=1/k^2$. Note that for $0\le s \le 1$, $Z_{s}$ dominates stochastically an exponential random variable with mean $2s$. Indeed, $Z_s$ is greater than the local time accumulated by $B$ at level $s$ after $T_1$, the hitting time of 1 by $B$. Since $(B_{T_1+t}, t \ge 0)$ has the distribution of a Brownian motion started at 1 killed upon hitting zero, we may
apply again Proposition 4.6 in Chapter VI of \cite{revuz-yor}, the claim follows. We deduce that
\begin{align*}
\sum_{k=2}^\infty \E(Ce^{-k^2\lambda Z_{1-u_k}/2}) & \le C\sum_{k=2}^\infty \int_0^\infty \frac{e^{- \lambda k^2 x} e^{ - x/(2-2/k^2) }}{ 2-2/k^2} dx\\
& \le C \sum_{k=2}^\infty \frac1{\lambda k^2(2-2/k^2) +1} <\infty.
\end{align*}
Thus by the Borel-Cantelli lemma, we get that
\begin{equation}\label{cd}
X_{u_k} \to 1 \text{ almost surely as }k \to \infty.
\end{equation}
Now, to obtain almost sure convergence for other values of $u$, let $k$ be such that $u_{k+1} < u \le u_k$, and consider the process $t \mapsto N(t,2ux), t \in [u_{k+1}, u_k].$ Note that, for a given size $y\ge 0$ say, the difference $|N(u,y)- N(u_k,y)|$ is bounded by
the total number of excursions that coalesce during the interval $(u_{k+1},u_k]$. To see this observe that $u \mapsto N(u,y)$ evolves by jumps of size 1 (either two excursions that had masses smaller than $y$ coalesce to give birth to an excursion of mass at least $y$ and the jump is positive, or  two excursions of masses larger than $y$ coalesce, in which case $N(u,y)$ decreases by 1). Hence for each coalescence event in $(u_{k+1},u_k]$ the process $N(u,y)$ changes by at most 1.

Thus we have that for all $u \in [u_{k+1}, u_k]$:
$$
|N(u,2ux)-N(u_k,2ux)|  \le |N(u_{k+1})-N(u_k)|.
$$
Since $N(u,x)$ is monotone in $x$, we obtain that for every $u \in (u_{k+1},u_k]$,
\begin{align*}
|N(u,2ux)- N(u_k,2u_kx)| &\le |N(u_{k+1})- N(u_k)| + |N(u_k,2ux)-N(u_k,2u_kx)| \\
& \le |N(u_{k+1})- N(u_k)| + |N(u_k,2u_{k+1}x)-N(u_k,2u_kx)| \\
& \le 2|N(u_{k+1})- N(u_k)| + |N(u_{k+1},2u_{k+1}x)-N(u_k,2u_kx)| \\
& = \Delta_k,
\end{align*}
say, where we have set
\begin{equation}\label{delta}
\Delta_k:=2|N(u_{k+1})- N(u_k)| + |N(u_{k+1},2u_{k+1}x)-N(u_k,2u_kx)|.
\end{equation}
Multiplying by $u$ and letting $X'_u=uN(u,2ux)$ (so that (\ref{cd0}) is equivalent to $X'_{u} \to Z_1e^{-x}/2$ almost surely as $u \to 0$ for all $x \ge 0$), we get:
\begin{align}
\nonumber \sup_{u \in (u_{k+1}, u_k]}|X'_u - X'_{u_k}| &= \sup_{u \in (u_{k+1}, u_k]} |u(N(u,2ux) - N(u_k,2u_kx)) \\ &\qquad \qquad \qquad+(u-u_k)N(u_k,2u_kx)| \nonumber \\
&\le  \sup_{u \in (u_{k+1}, u_k]} \{u \Delta_k + |u-u_k|N(u_k,2u_kx) \}\nonumber \\
& \le u_k \Delta_k + \frac{|u_{k+1}-u_k|}{u_k}X'_{u_k}. \label{cd1}
\end{align}
It is plain that the second term on the right-hand side converges almost surely to 0. To see that $u_k\Delta_k \to 0$ as well, it suffices to observe that by 
(\ref{cd}) applied respectively with $x=0$ and $x>0$, we see that $N(u_k) \sim C_1 k^2$ while $N(u_k, 2u_k x) \sim C_1 k^2$ for some random $C_1, C_2>0$. It follows that both terms in the right-hand side of (\ref{delta}) are $o(k^2)$, i.e., $u_k \Delta_k \to 0$ almost surely.
Thus the left-hand side of (\ref{cd1}) converges to 0 as well, and this implies $X'_u \to Z_1 e^{-x}/2$ almost surely as $u \to 0$. This finishes the proof of the lemma.
\end{proof}

We trivially obtain from Lemma \ref{L:exc} a result first derived by D. Aldous in \cite{aldous} (see his equation (35)). Let $K(t,x)$ be the number of blocks in Kingman's coalescent at time $t$ that are greater than $x$.

\begin{lemma}
For every $x \ge 0$, we have the almost sure convergence as $t \to 0$:
\begin{equation}\label{ald-est}
\frac{t}2K(t,tx/2) \longrightarrow e^{-x}.
\end{equation}
\end{lemma}

\begin{proof}
This is a trivial consequence of Theorem \ref{T1} and Lemma \ref{L:exc}. Indeed, since every block at time $t$ corresponds to an excursion above level $u$ with $u=U(t)$ by Theorem 1, and since the mass of a block is given by the renormalized amount of local time it accumulates at level 1, one may write:
$$
K(t,tx/2)= N(u,y)
$$
where $u = 1-U(t) \sim tZ_1 /4$ as $t \to 0$, and $y =tx Z_1/2$. Therefore, using the almost sure convergence result in Lemma \ref{L:exc} and making the necessary cancellations, we obtain the desired estimate (\ref{ald-est}).
\end{proof}

Corollary \ref{C:typ size} follows directly from (\ref{ald-est}), since if $B(t)$ denotes the mass of a block randomly chosen among the $K(t,0)$ blocks present at time $t$ (uniformly at random), then $\PP(B(t) > tx/2) = \E( K(t,tx/2)/K(t,0)) \longrightarrow e^{-x}$ by the Lebesgue dominated convergence theorem. It follows that:
$$
\frac{2B(t)}{t} \overset{d}\longrightarrow E
$$
where $E$ is an exponentially distributed random variable with mean 1.
Now, observe that the distribution of the frequency $F(t)$ of the block containing 1 is nothing but a size-biased version of the law of $B(t)$, that is, for every nonnegative Borel function $f$, we have (denoting $B=B(t)$ and $F=F(t)$):
$$
E\{f(F)\}= E\{Bf(B)\}.
$$
By considering for instance for every $a>0$, $\PP(F\le a) = \E(B\indic{B\le a})$, and the Lebesgue convergence theorem, we conclude that as $t \to 0$,
$$\frac{F}{2t} \overset{d}\longrightarrow \hat E$$
where $\hat E$ has a size-biased exponential distribution. That is, $\hat E$ has the distribution $xe^{-x}dx$ which is a Gamma(2) distribution. Equivalently, $\hat E$ is the sum of two standard exponential random variables $E+E'$. This concludes the proof of Corollary \ref{C:typ size}.

\section{Proof of Theorem \ref{T:MF spectra}}\label{S:preuve du thm}

The main ingredient for the proof of Theorem \ref{T:MF spectra} is the definition of \emph{reduced tree} associated with our
Brownian excursion $(B_s, 0 \le s \le \zeta)$, and results about the multifractal
spectrum of the Branching measure of Galton Watson trees due to
M\"orters and Shieh\cite{ms02}. Formally, the reduced tree $\mathbb{T}$ can be described by saying that $(B_s, 0 \le s \le \zeta)$ encodes a continuum random tree $\mathcal{T}$ with a metric $d$ and a root $o$ as in \cite{aldous1}. Each vertex $z \in \mathcal{T}$ has a unique geodesic $\gamma_z:[0,1] \to \mathcal{T}$ that connects it to $o$, such that if $d(o,z)=\rho$, then $d(o, \gamma_z(t)) = \rho t$ for all $0\le t \le 1$. We define $\mathbf{T}$ by
$$
\mathbf{T}= \bigcup_{z \in \mathcal{T}: d(z,0)=1} \gamma_z([0,1]).
$$
Informally, $\mathbf{T}$ is a continuum random tree obtained from $\mathcal{T}$ by taking away (``pruning") every vertex whose descendence does not reach distance 1 from the root. Equivalently, this is the tree which, at level $0\le u<1$, has as many branches as there are excursions above level $u$ that reach level 1. Thus if we let $|\mathbf{T}(u)|$ be the number of branches of $\mathbb{T}$ at level $0\le u<1$, then we have by definition $|\mathbf{T}(u)|= N(u)$, the number of excursions above $u$ reaching level 1.

The process $|\mathbf{T}(u)|$ is a variant of a process already considered by Neveu and
Pitman in a seminal paper \cite{neveu pitman}. The key observation is that:
\begin{equation}\label{yule1}
\{|\mathbf{T}(1-e^{-t})|, t \ge 0\} \overset{d}= \{Y_{t}, t \ge 0\},
\end{equation}
where $\{Y_t , t \ge 0\}$ is a (rate 1) Yule process. This is a continuous-time Galton-Watson process where individuals die at rate 1 to give birth to exactly two offsprings. In fact, an even stronger property holds. Consider the random tree $\mathbb{T}$ obtained from $\mathbf{T}$ by applying the same exponential time-change as in (\ref{yule1}). That is, for all $z \in \mathcal{T}$ with $d(z,o)=1$, define $\gamma'_z(t)=\gamma_z(1-e^{-t})$. Let
\begin{equation}
\label{yule2}
\mathbb{T}:= \bigcup_{z \in \mathcal{T}: d(z,0)=1} \gamma'_z([0,\infty)).
\end{equation}
Then $\mathbb{T}$ is a Yule tree. We can thus define the boundary $\partial \mathbb{T}$ of the tree $\mathbb{T}$ by taking $\partial \mathbb{T}$ to be the set of \emph{rays}, i.e., the set of all non-backtracking $\mathbb{T}$-valued paths $(\zeta(t), t \ge 0)$ such that $\zeta(t)$ is at distance $t$ from the root. The boundary $\partial \mathbb{T}$ is naturally equipped with a measure $\mu$, called the \emph{branching measure}, which is defined as follows. Thanks to \emph{Kesten-Stigum} theorem:
\begin{equation}\label{ks}
e^{-t} |\mathbb{T}(t)| \overset{d}\longrightarrow W
\end{equation}
where $W>0$ almost surely. In the Yule case, $W$ is well known to be an exponential random variable with mean 1.
 We next define a metric $\delta$ on $\bd$ by declaring that for $\zeta, \zeta' \in \bd$, $\delta(\zeta, \zeta') = e^{-t}$ where $t$ is the time of the ``most recent common ancestor", that is, $t= \sup\{s \ge 0 : \zeta(s)= \zeta'(s)\}.$ ($\delta$ is the so-called genealogical metric on $\bd$). A ball $B$ of radius $t$ for $\delta$ consists of all rays which pass through a given vertex $z \in \mathbb{T}$, at distance $t$ from the root. The subtree containing $z$ is itself a Yule tree and hence the Kesten-Stigum theorem applies to it, let $W(z)$ be the associated Kesten-Stigum random variable as in (\ref{ks}). Then define
 $$
\mu(B)= e^{-t} W(z).
 $$
It is easy to see that $\mu$ satisfies the assumptions of Carath\'eodory's extension theorem (see, e.g., \cite{durrett}, p. 444) and hence defines a finite measure on $\bd$, with total mass $\mu(\bd) = W(o)$. It will also be convenient to define the probability measure $ \mu^\sharp (\cdot) = \mu (\cdot)/ \mu(\bd)$.

Then our key claim is that we can identify $(\bd, \mu^\sharp)$ with Evans' metric space $(S, \eta)$:
\begin{equation}\label{id}
(\bd, \mu^\sharp) \rightleftharpoons (S, \eta),
\end{equation}
in the sense that we can find a continuous one-to-one map $\Phi: S \to \bd$ such that if $x \in S$ and $\zeta= \Phi(x)$,
\begin{equation}\label{id2}
\eta\{B(x,t)\} = \mu^\sharp \{ B(\zeta, 1-U(t))\},
\end{equation}
where $U(t)$ is the time-change appearing in Theorem \ref{T1}. To do this, we first define $\Phi$ on the integers. For $i \in \N$, simply define $\Phi(i)$ to be the ray associated with excursion $\eps_i$, where $\eps_1, \eps_2, \ldots$ are the excursions of $B$ above level 1, ordered by their respective heights. The definition of $\Phi$ is easily extended to $x \in S$ by taking a suitable sequence $x_n \in \N$ with $x_n \to x$ (see the proof of Theorem 5 in \cite{bbs2} for further details).

Having made the identification (\ref{id2}), it turns out that Theorem \ref{T:MF spectra} is now an easy application of Theorem 1.2 in \cite{ms02} and Theorem 1 in \cite{ms04}. These results are derived for a discrete-time  Galton-Watson tree $T$, and state the following. Assume that $T$ is a discrete-time Galton-Watson tree and that the number of offsprings $L $ of an individual is such that for some $r>0$, $\E(\exp(rL)) = \infty$ but $\E(\exp(tL)) < \infty$ if $t < r$, and $L$ is unbounded. Then we have, by Theorem 1.2 in \cite{ms02} (corresponding to the thick part of the spectrum), for all $0\le \theta\le a/r$:
\begin{equation}\label{spec1}
\dim\left\{ \xi \in \bd: \limsup_{n \to \infty}\frac{\mu(B(\xi, e^{-n}))}{m^{-n} n} = \theta\right\} =a- r\theta,
\end{equation}
where $a= \log \E(L)$ and $m= \E(L).$ Moreover when $\theta= a/r$ the above set is non-empty (see Lemma 3.3 (ii) in \cite{ms02}). Assuming further that $\PP(L= 1)>0 $ (the \emph{Schr\"oder case}) and letting $\tau= -\log \PP(L=1) /a$ we have, by Theorem 1 in \cite{ms04}, (corresponding to the thin part of the spectrum), for all $a\le \theta \le a(1+1/\tau)$:
\begin{equation}\label{spec2}
\dim\left\{ \xi \in \bd: \limsup_{n \to \infty}\frac{-\log \mu(B(\xi, e^{-n}))}{ n} = \theta\right\} = a(\frac{a}\theta(1+\tau)-\tau).
\end{equation}
Moreover when $\theta= a(1+1/\tau)$, the above set is non-empty almost surely.
To use these results in our case, let $T$ be the discrete-time tree obtained by sampling $\mathbb{T}$ at discrete times $1,2,\ldots$, so $T$ is a discrete-time Galton-Watson process which belongs to the Schr\"oder class. The random variable $W=W(o)$ is of course unchanged so it is an exponential random variable with mean 1 (note that $W$ is only equal to half the local time $Z_1$, corresponding to a ``one-sided" L\'evy approximation of $Z_1$ from below: see, e.g., (1.11) and (1.19) in Chapter VI of \cite{revuz-yor}). The distribution of $L$ is not particularly nice to write down but it is unbounded and we may nonetheless identify the parameters $a$, $r$ and $\tau$ as follows. Note first that since $W$ is exponential with parameter 1, $r=1$. Moreover, since
$$\{e^{-t} Y_t,t\ge 0\} \text{ is a martingale }$$
we obtain
$m=\E(Y_1)=\E(L)=e$, so $a=\log \E(L)=1$. We obtain $\tau$ by computing $\PP(L=1) = e^{-1}$ since every individual branches at rate 1, and hence $\tau =1$ as well. Using (\ref{id2}) together with (\ref{spec1}) and (\ref{spec2}), it is now straightforward to deduce Theorem \ref{T:MF spectra}. The details are left to the reader and are similar (in fact, much easier) than the proof of Lemma 26 in \cite{bbs2}. Note in particular that, since the construction of Kingman's coalescent in Theorem 1 holds at a fixed deterministic level 1, we avoid the use of Lemma 24 in \cite{bbs2}.

\section*{Acknowledgements}

We thank Christina Goldschmidt for sharing with us her thoughts about the Martin boundary of Galton-Watson trees, and we thank Ed Perkins for useful discussions. This work started when J.B. was at Universit\'e de Provence in Marseille, and N.B. was a postdoc at University of British Columbia. J.B. thanks the Math department of UBC for their invitation, during which this work started.

\end{document}

%% file: construction.pstex_t
\begin{picture}(0,0)%
\includegraphics{construction.pstex}%
\end{picture}%
\setlength{\unitlength}{2901sp}%
\begingroup\makeatletter\ifx\SetFigFontNFSS\undefined%
\gdef\SetFigFontNFSS#1#2#3#4#5{%
  \reset@font\fontsize{#1}{#2pt}%
  \fontfamily{#3}\fontseries{#4}\fontshape{#5}%
  \selectfont}%
\fi\endgroup%
\begin{picture}(9405,5085)(2791,-5686)
\put(8506,-3436){\makebox(0,0)[lb]{\smash{{\SetFigFontNFSS{8}{9.6}{\rmdefault}{\mddefault}{\updefault}{\color[rgb]{0,0,0}$e_1$}%
}}}}
\put(5941,-3436){\makebox(0,0)[lb]{\smash{{\SetFigFontNFSS{8}{9.6}{\rmdefault}{\mddefault}{\updefault}{\color[rgb]{0,0,0}$e_2$}%
}}}}
\put(6031,-1771){\makebox(0,0)[lb]{\smash{{\SetFigFontNFSS{8}{9.6}{\rmdefault}{\mddefault}{\updefault}{\color[rgb]{0,0,0}$\varepsilon_3$}%
}}}}
\put(5536,-2041){\makebox(0,0)[lb]{\smash{{\SetFigFontNFSS{8}{9.6}{\rmdefault}{\mddefault}{\updefault}{\color[rgb]{0,0,0}$\varepsilon_4$}%
}}}}
\put(9586,-1726){\makebox(0,0)[lb]{\smash{{\SetFigFontNFSS{8}{9.6}{\rmdefault}{\mddefault}{\updefault}{\color[rgb]{0,0,0}$\varepsilon_2$}%
}}}}
\put(3736,-3526){\makebox(0,0)[lb]{\smash{{\SetFigFontNFSS{8}{9.6}{\rmdefault}{\mddefault}{\updefault}{\color[rgb]{0,0,0}$u$}%
}}}}
\put(3736,-2491){\makebox(0,0)[lb]{\smash{{\SetFigFontNFSS{8}{9.6}{\rmdefault}{\mddefault}{\updefault}{\color[rgb]{0,0,0}1}%
}}}}
\put(8551,-2266){\makebox(0,0)[lb]{\smash{{\SetFigFontNFSS{8}{9.6}{\rmdefault}{\mddefault}{\updefault}{\color[rgb]{0,0,0}$\varepsilon_1$}%
}}}}
\put(6616,-2356){\makebox(0,0)[lb]{\smash{{\SetFigFontNFSS{8}{9.6}{\rmdefault}{\mddefault}{\updefault}{\color[rgb]{0,0,0}$\varepsilon_6$}%
}}}}
\put(6871,-2440){\makebox(0,0)[lb]{\smash{{\SetFigFontNFSS{8}{9.6}{\rmdefault}{\mddefault}{\updefault}{\color[rgb]{0,0,0}$\varepsilon_5$}%
}}}}
\end{picture}%

%% file: event_AB.pstex_t
\begin{picture}(0,0)%
\includegraphics{event_AB.pstex}%
\end{picture}%
\setlength{\unitlength}{2486sp}%
\begingroup\makeatletter\ifx\SetFigFontNFSS\undefined%
\gdef\SetFigFontNFSS#1#2#3#4#5{%
  \reset@font\fontsize{#1}{#2pt}%
  \fontfamily{#3}\fontseries{#4}\fontshape{#5}%
  \selectfont}%
\fi\endgroup%
\begin{picture}(9972,3726)(391,-4414)
\put(3331,-4291){\makebox(0,0)[lb]{\smash{{\SetFigFontNFSS{7}{8.4}{\rmdefault}{\mddefault}{\updefault}{\color[rgb]{0,0,0}$T_M$}%
}}}}
\put(1036,-4336){\makebox(0,0)[lb]{\smash{{\SetFigFontNFSS{7}{8.4}{\rmdefault}{\mddefault}{\updefault}{\color[rgb]{0,0,0}$T^1_\delta$}%
}}}}
\put(496,-3481){\makebox(0,0)[lb]{\smash{{\SetFigFontNFSS{7}{8.4}{\rmdefault}{\mddefault}{\updefault}{\color[rgb]{0,0,0}$\delta$}%
}}}}
\put(496,-2536){\makebox(0,0)[lb]{\smash{{\SetFigFontNFSS{7}{8.4}{\rmdefault}{\mddefault}{\updefault}{\color[rgb]{0,0,0}$x$}%
}}}}
\put(496,-2221){\makebox(0,0)[lb]{\smash{{\SetFigFontNFSS{7}{8.4}{\rmdefault}{\mddefault}{\updefault}{\color[rgb]{0,0,0}$y$}%
}}}}
\put(406,-1141){\makebox(0,0)[lb]{\smash{{\SetFigFontNFSS{7}{8.4}{\rmdefault}{\mddefault}{\updefault}{\color[rgb]{0,0,0}$m$}%
}}}}
\put(4141,-4291){\makebox(0,0)[lb]{\smash{{\SetFigFontNFSS{7}{8.4}{\rmdefault}{\mddefault}{\updefault}{\color[rgb]{0,0,0}$T^2_\delta$}%
}}}}
\put(7426,-4291){\makebox(0,0)[lb]{\smash{{\SetFigFontNFSS{7}{8.4}{\rmdefault}{\mddefault}{\updefault}{\color[rgb]{0,0,0}$T_M$}%
}}}}
\put(6121,-4291){\makebox(0,0)[lb]{\smash{{\SetFigFontNFSS{7}{8.4}{\rmdefault}{\mddefault}{\updefault}{\color[rgb]{0,0,0}$T^1_\delta$}%
}}}}
\put(5671,-2536){\makebox(0,0)[lb]{\smash{{\SetFigFontNFSS{7}{8.4}{\rmdefault}{\mddefault}{\updefault}{\color[rgb]{0,0,0}$x$}%
}}}}
\put(5671,-2176){\makebox(0,0)[lb]{\smash{{\SetFigFontNFSS{7}{8.4}{\rmdefault}{\mddefault}{\updefault}{\color[rgb]{0,0,0}$y$}%
}}}}
\put(5581,-1141){\makebox(0,0)[lb]{\smash{{\SetFigFontNFSS{7}{8.4}{\rmdefault}{\mddefault}{\updefault}{\color[rgb]{0,0,0}$m$}%
}}}}
\put(8461,-4291){\makebox(0,0)[lb]{\smash{{\SetFigFontNFSS{7}{8.4}{\rmdefault}{\mddefault}{\updefault}{\color[rgb]{0,0,0}$T^2_\delta$}%
}}}}
\put(5671,-3481){\makebox(0,0)[lb]{\smash{{\SetFigFontNFSS{7}{8.4}{\rmdefault}{\mddefault}{\updefault}{\color[rgb]{0,0,0}$\delta$}%
}}}}
\put(2476,-871){\makebox(0,0)[lb]{\smash{{\SetFigFontNFSS{7}{8.4}{\rmdefault}{\mddefault}{\updefault}{\color[rgb]{0,0,0}$A_\delta$}%
}}}}
\put(7606,-871){\makebox(0,0)[lb]{\smash{{\SetFigFontNFSS{7}{8.4}{\rmdefault}{\mddefault}{\updefault}{\color[rgb]{0,0,0}$B_\delta$}%
}}}}
\end{picture}%

%% file: kingBM12.bbl
\begin{thebibliography}{99}

\bibitem{aldous1} D. Aldous (1991). The continuum random tree I. \emph{Ann. Probab.}
\textbf{19}, 1-28.



\bibitem{aldous3} D. Aldous (1993). The continuum random tree III.
\emph{Ann. Probab.} \textbf{21}, 248-289.

\bibitem{aldouscond} D. J. Aldous (1998). Brownian excusion conditioned on its local time. \emph{Elect. Comm. in Probab.}, 3, 79--90.


\bibitem{aldous} D. Aldous (1999). Deterministic and stochastic models
for coalescence (aggregation and coagulation): a review of the
mean-field theory for probabilists. \emph{Bernoulli}, 5: 3--48.



\bibitem{bbl2} J. Berestycki, N. Berestycki and V. Limic.
(2008).  Interpreting $\La$-coalescent speed of coming down from infinity
via particle representation of super-processes.
In preparation.


\bibitem{bbs1} J. Berestycki, N. Berestycki and J. Schweinsberg
(2005). Small-time behavior of beta-coalescents.
\emph{Ann. Inst. H. Poincar\'e (B) Probabilit\'es et Statistiques}, 44:2, 214-238

\bibitem{bbs2} J. Berestycki, N. Berestycki and J. Schweinsberg
(2007). Beta-coalescents and continuous stable random trees. \emph{Ann. Probab.}   {\bf 35}, 1835-1887.

\bibitem{bertoin} J. Bertoin (2006). \emph{Random Fragmentation and
    Coagulation Processes.} Cambridge Studies in Advanced Mathematics.


\bibitem{bg} J. Bertoin and C. Goldschmidt (2004). Dual random fragmentation and coagulation and an application to the genealogy of Yule processes. In \emph{Mathematics and Computer Science III: Algorithms, Trees, Combinatorics and Probabilities}. M. Drmota, P. Flajolet, D. Gardy, B. Gittenberger (Eds.), 295--308.

\bibitem{duleg} T. Duquesne and J.-F. Le Gall (2002). {\it Random Trees,
L\'evy Processes, and Spatial Branching Processes.  Ast\'erisque}
{\bf 281}.

\bibitem{durrett} R. Durrett (2004). \emph{Probability: theory and
examples}. Duxbury advanced series, $3^{\text{rd}}$ edition.


\bibitem{etheridge} A. Etheridge (2000). \emph{An Introduction to
Superprocesses}. American Mathematical Society, University Lecture
series, 20.


\bibitem{evans}S. N. Evans. (2000).  Kingman's coalescent as a random
metric space.  In {\it Stochastic Models: A Conference in Honour of
Professor Donald A. Dawson} (L. G. Gorostiza and B. G. Ivanoff eds.)
Canadian Mathematical Society/American Mathematical Society.

\bibitem{cg} C. Goldschmidt. Private communication.

\bibitem{king82}J. F. C. Kingman (1982).  The coalescent.
{\it Stoch. Proc. Appl.} {\bf 13}, 235--248.


\bibitem{legall93} J.-F. Le Gall (1993). The uniform random tree in the Brownian excurion. \emph{Prob. Th. Rel. Fields}, 96, 369--383.


\bibitem{lyons} R. Lyons (1990). Random walks and percolation on trees. \emph{Ann. Probab.}, 18, 931--958.

\bibitem{ms02} P. M\"orters and N.-R. Shieh (2002). Thin and thick points for branching measure on a Galton-Watson tree. \emph{Statistics and Probability Letters} 58, 13--22.

\bibitem{ms04} P. M\"orters and N.-R. Shieh (2004). On the multifractal spectrum for branching measure on a Galton-Watson tree. \emph{Journal of Applied Probability} 41, 1223--1229.

\bibitem{neveu pitman} J. Neveu and J. Pitman (1989).  The branching
process in a Brownian excursion. S\'eminaire de Probabilit\'es XXIII,
\emph{Lecture Notes Math.} 1372:248--257, Springer.


\bibitem{perkins91} E. A. Perkins (1991). Conditional Dawson-Watanabe superprocesses and Fleming-Viot process. \emph{Seminar on Stochastic Processes}. Birkhauser.


\bibitem{revuz-yor} D. Revuz and M. Yor. \emph{Continuous martingales and Brownian motion}. Grundlehren der mathematischen Wissenschaften, Vol. 293, Springer. 3rd edition, 1999.


\bibitem{warren} J. Warren (1999). On a result of David Aldous concerning the trees in the conditioned excursion. \emph{Elect. Comm. in Probab.} 4, 25--29.

\bibitem{warrenyor} J. Warren and M. Yor (1998). The Brownian burglar: conditioning Brownian motion on its local time process. In J. Az\'ema, M. Emery, M. Ledoux and M. Yor, eds. \emph{S\'eminaire de Probabilit\'es XXXII} 328--342, Springer, Lecture Notes in Math. 1686.



\end{thebibliography}
